\documentclass[secthm,seceqn,amsthm,ussrhead,12pt]{amsart}
\usepackage[utf8]{inputenc}
\usepackage[english]{babel}
\usepackage{amssymb,amsmath,amsthm,amsfonts,xcolor,enumerate,hyperref,comment,longtable,cleveref}

\usepackage{times}
\usepackage{cite}
\usepackage{pdflscape}
\usepackage{ulem}
\usepackage[mathcal]{euscript}
\usepackage{tikz}
\usepackage{hyperref}
\usepackage{cancel}
\usepackage{stmaryrd}

\usetikzlibrary{arrows}

\newcommand{\Co}{{\mathbb C}}

\newcommand{\cc}{{\mathfrak C}}

\newtheorem{theorem}{Theorem}
\newtheorem{lemma}[theorem]{Lemma}

\usepackage{stmaryrd}
\usepackage{xcolor}

\newcommand{\der}[1]{\operatorname{\mathrm{Der}}{#1}}

\newcommand{\gdim}{\mathrm{gdim}}

\newcommand{\A}{\mathbf{A}}
\newcommand{\B}{\mathbf{B}}

\newcommand{\af}{\alpha}

\begin{document}
\noindent{\Large
The geometric classification of nilpotent commutative $\mathfrak{CD}$-algebras}

   \

   {\bf  Doston Jumaniyozov,
   Ivan   Kaygorodov \&
   Abror Khudoyberdiyev}

\

\

\

\noindent{\bf Abstract}: 
{\it We give a geometric  classification of complex
$5$-dimensional nilpotent commutative  $\mathfrak{CD}$-algebras. 
 The corresponding geometric variety has dimension $24$ and decomposes into $10$ irreducible components determined by the Zariski closures of 
 a two-parameter family of algebras,  
 three one-parameter families of algebras, 
 and $6$ rigid algebras.}

\

\noindent {\bf Keywords}: {\it Nilpotent algebra, Jordan algebra,
commutative $\mathfrak{CD}$-algebra, geometric classification, degeneration.}

\ 

    \noindent {\bf MSC2020}: 	17A30, 17D99, 17B30.

\section*{Introduction}

There are many results related to the algebraic and geometric 
classification
of low-dimensional algebras in the varieties of Jordan, Lie, Leibniz and 
Zinbiel algebras;
for algebraic classifications  see, for example, 
\cite{ack,          ikm19,  kkk18}; 
for geometric classifications and descriptions of degenerations see, for example, 
\cite{ack,   bb14, fkkv19,   jkk19, ikm19,    kkk18, kama, kk20g, kkl20, klp19,   maz79, S90}.
In this present paper, we give a geometric classification of nilpotent commutative $\mathfrak{CD}$-algebras.
 This is a new class of non-associative algebras introduced in \cite{ack,kps19}.
The idea of the definition of a $\mathfrak{CD}$-algebra comes from the following property of Jordan and Lie algebras: {\it the commutator of any pair of multiplication operators is a derivation}.
 This gives three identities of degree four,  which reduce to only one identity of degree four in the commutative or anticommutative case.
Commutative and anticommutative  $\mathfrak{CD}$-algebras are related to many interesting varieties of algebras.
 Thus, anticommutative  $\mathfrak{CD}$-algebras is a generalization of Lie algebras, 
 containing the intersection of Malcev and Sagle algebras as a proper subvariety.   Moreover, the following intersections of varieties coincide:
Malcev and Sagle algebras; 
Malcev and anticommutative  $\mathfrak{CD}$-algebras; 
Sagle and anticommutative  $\mathfrak{CD}$-algebras.
On the other hand, 
the variety of anticommutative  $\mathfrak{CD}$-algebras is a proper subvariety   of 
the varieties of binary Lie algebras 
and almost Lie algebras \cite{kz}.
The variety of anticommutative  $\mathfrak{CD}$-algebras coincides with the intersection of the varieties of binary Lie algebras and almost Lie algebras.
Commutative  $\mathfrak{CD}$-algebras is a generalization of Jordan algebras, 
which is a generalization of associative commutative algebras.
On the other hand, the variety of commutative  $\mathfrak{CD}$-algebras is also known as the variety of almost-Jordan algebras, which states in the bigger variety of generalized almost-Jordan algebras \cite{arenas,hl,labra}.
 The $n$-ary  version of commutative  $\mathfrak{CD}$-algebras was introduced in a recent paper by 
Kaygorodov, Pozhidaev and Saraiva \cite{kps19}. 

The variety of almost-Jordan algebras is the variety of commutative algebras, 
satisfying \[2((yx)x)x+yx^3=3(yx^2)x.\] 
This present identity appeared in a paper of Osborn \cite{os65},
during the study of identities of degree less than or equal to $4$ of non-associative algebras. The identity is a linearized form of the Jordan identity.
The systematic study of almost-Jordan algebras was initiated in the next paper of Osborn \cite{osborn65} and it was continued in some papers of Petersson \cite{petersson, petersson67}, Osborn \cite{osborn69}, and Sidorov \cite{Sidorov_1981}
(sometimes, it was  called as Lie triple algebras).
Hentzel and  Peresi proved that every semiprime almost-Jordan ring is Jordan \cite{peresi}.
After that, 
Labra and Correa
proved  that a finite-dimensional almost-Jordan right-nilalgebra is nilpotent \cite{cl09,cl09-2}.
Assosymmetric algebras under the symmetric product give almost-Jordan algebras  \cite{askar18}.

\medskip 

\paragraph{\bf Motivation and contextualization} 
Geometric properties of a variety of algebras have been an object of study since 1970's. Gabriel~\cite{gabriel} described the irreducible components of the variety of $4$-dimensional unital associative algebras. Mazzola~\cite{maz79} classified algebraically and geometrically the variety of unital associative algebras of dimension $5$.  Burde and Steinhoff~\cite{BC99} constructed the graphs of degenerations for the varieties of $3$-dimensional and $4$-dimensional Lie algebras over $\Co$. Grunewald and O'Halloran~\cite{GRH} calculated the degenerations for the nilpotent Lie algebras of dimension up to $5$. Seeley~\cite{S90} solved the same problem for $6$-dimensional complex nilpotent Lie algebras. 
Chouhy~\cite{chouhy} proved that, in the case of finite-dimensional associative algebras,
 the $N$-Koszul property is preserved under the degeneration relation.
Degenerations have also been used to study a level of complexity of an algebra (see~\cite{g93,wolf1,wolf2,  kh15}).
Given algebras ${\bf A}$ and ${\bf B}$ in the same variety, we write ${\bf A}\to {\bf B}$ and say that ${\bf A}$ {\it degenerates} to ${\bf B}$, or that ${\bf A}$ is a {\it deformation} of ${\bf B}$, if ${\bf B}$ is in the Zariski closure of the orbit of ${\bf A}$ (under the base-change action of the general linear group). The study of degenerations of algebras is very rich and closely related to deformation theory, in the sense of Gerstenhaber \cite{ger63}. It offers an insightful geometric perspective on the subject and has been the object of a lot of research.
In particular, there are many results concerning degenerations of algebras of small dimensions in a  variety defined by a set of identities.
One of the main problems of the {\it geometric classification} of a variety of algebras is a description of its irreducible components. In the case of finitely-many orbits (i.e., isomorphism classes), the irreducible components are determined by the rigid algebras --- algebras whose orbit closure is an irreducible component of the variety under consideration. 
The algebraic classification of complex $5$-dimensional nilpotent commutative $\mathfrak{CD}$-algebras was obtained in \cite{jkk20}, and in the present paper we continue the study of the variety by giving its geometric classification.


\section{The algebraic classification of complex $5$-dimensional nilpotent commutative $\mathfrak{CD}$-algebras}

The algebraic classifiction of $5$-dimensional nilpotent commutative $\mathfrak{CD}$-algebras has three steps:
the classification of all associative commutative algebras was given by Mazzola in 1979 \cite{maz79};
the next step is the classification of all non-associative Jordan algebras was given by Hegazi and Abdelwahab in 2016 \cite{ha16};
and the last step is the classification of all non-Jordan commutative $\mathfrak{CD}$-algebras was given by 
Jumaniyozov,  Kaygorodov and   Khudoyberdiyev in  2021 \cite{jkk20}. 
Let us give the list of algebras from the last part from this long classification:

\begin{theorem}\label{teor}
Let $\mathfrak{C}$ be a complex $5$-dimensional nilpotent commutative $\mathfrak{CD}$-algebra.
Then $\mathfrak{C}$ is a Jordan algebra or it is isomorphic to one algebra from the following list:
{\tiny 
\begin{longtable}{lllllllllll}

$\mathfrak{C}^{5}_{01}$&$:$& $e_1 e_1 = e_2$ & $e_2 e_2=e_3$\\
			            
$\mathfrak{C}^{5}_{02}(\alpha)$&$:$&  $e_1 e_1 = e_2$  & $e_1 e_2=e_3$& $e_1e_3= \alpha e_4$  & \multicolumn{2}{l}{$e_2e_2= (\alpha +1)e_4$} \\
			
$\mathfrak{C}^{5}_{03}$& $: $&   $e_1 e_1 = e_2$& $e_1e_3=e_4$& $e_2e_2=e_4$		\\

$\mathfrak{C}^{5}_{04}$& $: $&     $e_1 e_1 = e_2$ & $e_2e_2=e_4$& $e_3e_3=e_4$ \\
 
$\mathfrak{C}^5_{05}$ & $: $ & $e_1 e_2 = e_3$ & $e_3 e_3=e_4$	\\  
$\mathfrak{C}^5_{06}$ & $: $ & $e_1 e_1 = e_4$ & $e_1 e_2=e_3$ & $e_2e_2=e_4$& $e_3e_3=e_4$\\ 
$\mathfrak{C}^5_{07}$ & $: $ & $e_1 e_1 = e_4$ & $e_1 e_2=e_3$ & $e_3e_3=e_4$\\
$\mathfrak{C}^5_{08}$ & $: $ & $e_1e_1=e_2$ & $e_1e_3=e_4$ & $e_2e_2=e_5 $\\
$\mathfrak{C}^5_{09}$ & $: $& $e_1e_1=e_2$ & $e_1e_3=e_4$ & $e_2e_2=e_5$ & $e_3e_3=e_5 $\\
$\mathfrak{C}^5_{10}$ & $ :$ & $e_1e_1=e_2$ & $e_1e_2=e_4$ & $e_2e_2=e_5 $\\
$\mathfrak{C}^5_{11}$ & $ : $ & $e_1e_1=e_2$ & $e_1e_2=e_4$ & $e_1e_3=e_5$ & $e_2e_2=e_5 $\\
$\mathfrak{C}_{12}^5(\alpha)$&$:$& 
$e_1e_1=e_2$ & $e_1e_2=e_3$ &$e_1e_3=(\alpha+1) e_5$ &$e_2e_2= \alpha e_5$  &$e_2e_4= e_5$\\
$\mathfrak{C}_{13}^5(\alpha, \beta)$&$:$& 
$e_1e_1=e_2$ & $e_1e_2=e_3$ &$e_1e_3=(\alpha+1) e_5$ &$e_2e_2= \alpha e_5$ &$e_2e_4= \beta e_5$ &$e_4e_4= e_5$\\
$\mathfrak{C}^5_{14}$ & $ : $ & $e_1e_1=e_2$ & $e_2e_2=e_5$ & $e_3e_3=e_4 $\\
$\mathfrak{C}^5_{15}$ & $ : $ & $e_1e_1=e_2$ & $e_1e_3=e_5 $ & $ e_2e_2=e_5$ & $e_3e_3=e_4 $\\
$\mathfrak{C}_{16}^5(\alpha)$& $ : $ & $e_1e_1=e_2$ & $e_1e_2=e_4$ &$e_1e_4= (\af+1) e_5$& $e_2e_2=\af e_5$ &$e_3e_3=e_4$\\
$\mathfrak{C}^5_{17}$ & $ : $ & $e_1e_1=e_2$ & $e_1e_2=e_4$ & $e_1e_3=e_5$ & $e_2e_2=e_5$ & $e_3e_3=e_4 $\\
$\mathfrak{C}^5_{18}$ & $ : $ & $e_1e_1=e_2$ & $e_2e_2=e_5$ & $e_2e_3=e_4 $\\
$\mathfrak{C}^5_{19}$ & $ : $ & $e_1e_1=e_2$ & $e_2e_2=e_5$ & $e_2e_3=e_4$ & $e_3e_3=e_5 $\\
$\mathfrak{C}^5_{20}$ & $ : $ & $e_1e_1=e_2$ & $e_1e_3=e_5$ & $e_2e_2=e_5$ & $e_2e_3=e_4 $\\
$\mathfrak{C}^5_{21}$ & $ : $ & $e_1e_1=e_2$ & $e_1e_3=e_5$ & $e_2e_2=e_5$ & $e_2e_3=e_4$ & $e_3e_3=e_5 $\\
$\mathfrak{C}^5_{22}$ & $ : $ & $e_1e_1=e_2$ & $e_1e_2=e_5$ & $e_2e_2=e_5$ & $e_2e_3=e_4 $\\
$\mathfrak{C}^5_{23}$ & $ : $ & $e_1e_1=e_2$ & $e_1e_2=e_5$ & $e_2e_2=e_5$ & $e_2e_3=e_4$ & $e_3e_3=e_5 $\\
$\mathfrak{C}^5_{24}(\alpha)$ & $ : $ & $e_1e_1=e_2$ & $e_1e_2=e_5$ & $e_1e_3=e_5$ & $e_2e_2=e_5$ & $e_2e_3=e_4$ & $e_3e_3=\alpha e_5 $\\
$\mathfrak{C}_{25}^5$&$:$& $e_1e_1=e_2$ & $e_1e_2=e_3$ &$e_1e_3=e_4$ & $e_2e_2=e_5$\\
$\mathfrak{C}^5_{26}(\alpha,\beta)$ & $: $ & $e_1e_1=\alpha e_5$ & $e_1e_2=e_3$ & $e_2e_2=\beta e_5$ & $e_1e_3=e_4+e_5$ & $e_2e_3=e_4$ & $e_3e_3=e_5 $\\
$\mathfrak{C}^5_{27}(\alpha)$ & $ : $ & $e_1e_1=\alpha e_5$ & $e_1e_2=e_3$ & $e_2e_2=e_5$ & $e_1e_3=e_4$ & $e_2e_3=e_4$ & $e_3e_3=e_5 $\\
$\mathfrak{C}^5_{28}$ & $ : $ & $e_1e_1=e_5$ & $e_1e_2=e_3$ & $e_1e_3=e_4$ & $e_2e_3=e_4$ & $e_3e_3=e_5 $\\
$\mathfrak{C}^5_{29}$ & $ : $ & $e_1e_2=e_3$ & $e_1e_3=e_4$ & $e_2e_3=e_4$ & $e_3e_3=e_5 $\\
$\mathfrak{C}^5_{30}(\alpha)$ & $ : $ & $e_1e_1=e_4+\alpha e_5$ & $e_1e_2=e_3$ & $e_2e_2=e_5$ & $e_2e_3=e_4$ & $e_3e_3=e_5 $\\
$\mathfrak{C}^5_{31}$ & $ : $ & $e_1e_1=e_4+e_5$ & $e_1e_2=e_3$ & $e_2e_3=e_4$ & $e_3e_3=e_5 $\\
$\mathfrak{C}^5_{32}$ & $ : $ & $e_1e_1=e_4$ & $e_1e_2=e_3$ & $e_2e_3=e_4$ & $e_3e_3=e_5 $\\
$\mathfrak{C}^5_{33}$ & $ : $ & $e_1e_1=e_5$ & $e_1e_2=e_3$ & $e_2e_2=e_5$ & $e_2e_3=e_4$ & $e_3e_3=e_5 $\\
$\mathfrak{C}^5_{34}$ & $ : $ & $e_1e_1=e_5$ & $e_1e_2=e_3$ & $e_2e_3=e_4$ & $e_3e_3=e_5 $\\
$\mathfrak{C}^5_{35}$ & $ :$ & $e_1e_2=e_3$ & $e_2e_2=e_5$ & $e_2e_3=e_4$ & $e_3e_3=e_5 $\\
$\mathfrak{C}^5_{36}$ & $ :$ & $e_1e_2=e_3$ & $e_2e_3=e_4$ & $e_3e_3=e_5 $\\
$\mathfrak{C}^5_{37}$ & $ :$ & $e_1e_1=e_4+e_5 $ & $e_1e_2=e_3$ & $e_2e_2=e_4$ & $e_3e_3=e_5 $\\
$\mathfrak{C}^5_{38}$ & $ :$ & $e_1e_1=e_4$ & $e_1e_2=e_3$ & $e_2e_2=e_4$ & $e_3e_3=e_5 $\\
$\mathfrak{C}^5_{39}$ & $ :$ & $e_1e_1=e_5$ & $e_1e_2=e_3$ & $e_2e_2=e_4$ & $e_3e_3=e_5 $\\
$\mathfrak{C}^5_{40}$ & $ : $ & $e_1e_2=e_3$ & $e_2e_2=e_4$ & $e_3e_3=e_5 $ \\
$\mathfrak{C}_{41}^5$& $ : $ & $e_1e_1=e_2$ & $e_2e_2=e_5$ &$e_3e_4=e_5$\\
$\mathfrak{C}_{42}^5$& $ : $ & $e_1e_1=e_2$ & $e_1e_3=e_5$ &$e_2e_2=e_5$ &$e_4e_4= e_5$\\
$\mathfrak{C}_{43}^5$& $ : $ & $e_1e_1=e_5$ &  $e_1e_2=e_3$ &$e_2e_4=e_5$ & $e_3e_3=e_5$\\
$\mathfrak{C}_{44}^5$& $ : $ &  $e_1e_1=e_5$ & $e_1e_2=e_3$ & $e_2e_2=e_5$ & $e_3e_3=e_5$& $e_4e_4=e_5$\\
$\mathfrak{C}_{45}^5$& $ : $ & $e_1e_1=e_5$ & $e_1e_2=e_3$ & $e_3e_3=e_5$ & $e_4e_4=e_5$\\
$\mathfrak{C}_{46}^5$& $ : $ & $e_1e_2=e_3$ & $e_1e_4=e_5$ & $e_2e_4=e_5$ & $e_3e_3=e_5$\\
$\mathfrak{C}_{47}^5$& $ : $ & $e_1e_2=e_3$ & $e_2e_4=e_5$ & $e_3e_3=e_5$ \\
$\mathfrak{C}_{48}^5$& $ : $ & $e_1e_2=e_3$ & $e_3e_3=e_5$ & $e_4e_4=e_5$\\
 $\mathfrak{C}_{49}^5(\af)$& $ : $ & $e_1e_1=e_3$ & $e_1e_2=e_5$ &$e_2e_2=e_4$ & $e_3e_3=\af e_5$ & $e_3e_4=e_5$& $e_4e_4=e_5$ \\
 $\mathfrak{C}_{50}^5$& $ : $ & $e_1e_1=e_3$  & $e_1e_2=e_5$ &$e_2e_2=e_4$ & $e_3e_3=e_5$ & $e_4e_4=e_5$\\
 $\mathfrak{C}_{51}^5$& $ : $ & $e_1e_1=e_3$  & $e_1e_2=e_5$ &$e_2e_2=e_4$ & $e_3e_4=e_5$\\
 $\mathfrak{C}_{52}^5(\af)$& $ : $ & $e_1e_1=e_3$  & $e_1e_3=\af e_5$ &$e_2e_2=e_4$ &$e_2e_3=e_5$ &$e_3e_3=e_5$
 &$e_3e_4=e_5$ &$e_4e_4=e_5$\\
 $\mathfrak{C}_{53}^5$& $ : $ & $e_1e_1=e_3$  & $e_1e_3=e_5$ &$e_2e_2=e_4$ & $e_2e_3=e_5$ & $e_4e_4=e_5$\\
 $\mathfrak{C}_{54}^5$& $ : $ & $e_1e_1=e_3$  & $e_1e_3=e_5$ &$e_2e_2=e_4$ & $e_4e_4=e_5$\\
 $\mathfrak{C}_{55}^5$& $ : $ & $e_1e_1=e_3$  & $e_2e_2=e_4$ &$e_2e_3=e_5$ & $e_4e_4=e_5$\\
 $\mathfrak{C}_{56}^5(\af)$& $ : $ & $e_1e_1=e_3$  &$e_2e_2=e_4$ & $e_3e_3=\alpha e_5$ & $e_3e_4=e_5$ & $e_4e_4=e_5$\\
 $\mathfrak{C}_{57}^5$& $ : $ & $e_1e_1=e_3$  &$e_2e_2=e_4$ & $e_3e_3=e_5$ & $e_4e_4=e_5$\\
 $\mathfrak{C}_{58}^5$& $ : $ & $e_1e_1=e_3$  &$e_2e_2=e_4$ & $e_3e_4=e_5$\\
 $\mathfrak{C}_{59}^5$& $ : $ & $e_1e_1=e_3$ & $e_1e_2=e_4$ &$e_1e_3=e_5$ &$e_2e_2=e_5$ &$e_4e_4=e_5$\\
$\mathfrak{C}_{60}^5$& $ : $ & $e_1e_1=e_3$ & $e_1e_2=e_4$ &$e_1e_3=e_5$ &$e_2e_3=e_5$ &$e_4e_4=e_5$\\
$\mathfrak{C}_{61}^5$& $ : $ & $e_1e_1=e_3$ & $e_1e_2=e_4$ &$e_1e_3=e_5$ &$e_4e_4=e_5$\\ 
$\mathfrak{C}_{62}^5(\af)$& $ : $ & $e_1e_1=e_3$ & $e_1e_2=e_4$&$e_1e_4=e_5$ &$e_2e_2=\af e_5$ &$e_3e_3=e_5$\\ 
$\mathfrak{C}_{63}^5$& $ : $ & $e_1e_1=e_3$ & $e_1e_2=e_4$ &$e_2e_2=e_5$ &$e_3e_4=e_5$\\
$\mathfrak{C}_{64}^5$& $ : $ & $e_1e_1=e_3$ & $e_1e_2=e_4$ &$e_2e_3=e_5$ &$e_3e_3=e_5$ &$e_4e_4=e_5$\\ 
$\mathfrak{C}_{65}^5$& $ : $ & $e_1e_1=e_3$ & $e_1e_2=e_4$ &$e_2e_3=e_5$ &$e_4e_4=e_5$\\
$\mathfrak{C}_{66}^5$& $ : $ & $e_1e_1=e_3$ & $e_1e_2=e_4$ &$e_2e_4=e_5$ &$e_3e_3=e_5$\\ 
$\mathfrak{C}_{67}^5$& $ : $ & $e_1e_1=e_3$ & $e_1e_2=e_4$ &$e_3e_3=e_5$ &$e_4e_4=e_5$\\ 
$\mathfrak{C}_{68}^5$& $ : $ &   $e_1e_1=e_3$ & $e_1e_2=e_4$ &$e_3e_4=e_5$\\
${\mathfrak{C}}_{69}^{5}(\af)$ & $ : $ &   $e_1 e_1 = e_4$&$e_1e_2=\af e_5$ &$e_1e_3=e_5$ &$e_2e_2=e_5$
&$ e_2 e_3=e_4$ &$e_4e_4=e_5$\\
${\mathfrak{C}}_{70}^{5}$ & $ : $ &   $e_1 e_1 = e_4$&$e_1e_2=e_5$ &$e_1e_3=e_5$ &$ e_2 e_3=e_4$ &$e_4e_4=e_5$ \\
${\mathfrak{C}}_{71}^{5}$& $ : $ &   $e_1 e_1 = e_4$&$e_1e_2=e_5$&$ e_2 e_3=e_4$ &$e_4e_4=e_5$ \\
${\mathfrak{C}}_{72}^{5}$ & $ : $ &   $e_1 e_1 = e_4$&$e_2e_2=e_5$&$ e_2 e_3=e_4+e_5$ &$e_4e_4=e_5$ \\
${\mathfrak{C}}_{73}^{5}$ & $ : $ &   $e_1 e_1 = e_4$&$e_2e_2=e_5$&$ e_2 e_3=e_4$ &$e_4e_4=e_5$ \\
${\mathfrak{C}}_{74}^{5}$ & $ : $ &   $e_1 e_1 = e_4$&$ e_2 e_3=e_4+e_5$ &$e_4e_4=e_5$ \\
${\mathfrak{C}}_{75}^{5}$ & $ : $ &   $e_1 e_1 = e_4$&$ e_2 e_3=e_4$ &$e_4e_4=e_5$ \\
$\mathfrak{C}_{76}^5$& $ : $ & $e_1e_1=e_2$ & $e_1e_2=e_4$  &$e_1e_4= e_5$ &$e_2e_2= - 2 e_5$ &$e_3e_3=e_4+3e_5$\\
$\mathfrak{C}_{77}^5$& $ : $ & $e_1e_1=e_2$ & $e_1e_2=e_4$ &$e_1e_4= e_5$ &$e_2e_3=  e_5$ &$e_3e_3=e_4$\\

$\mathfrak{C}_{78 }^5$& $ : $ & $e_1 e_1 = e_2$ & $e_1 e_2=e_3$ & $e_1e_3=e_4$ &$e_1e_4=e_5$&$e_2e_2=e_4+e_5$ &$e_2e_3=e_5$\\ 

$\mathfrak{C}_{79}^5 $& $ : $ & 
$e_1 e_1 = e_2$& 
$e_1 e_2=e_3$&
$e_1e_3=  e_4$& 
$e_1e_4=  e_5$& 
$e_2e_2= 2 e_4+e_5$&
$e_2e_3=  4 e_5$ \\

$\mathfrak{C}_{80}^5(\af)$& $ : $ & 
$e_1 e_1 = e_2$& 
$e_1 e_2=e_3$&
$e_1e_3= \alpha e_4$& 
$e_1 e_4=e_5$&
$e_2e_2= (\alpha +1)e_4$&
\multicolumn{2}{l}{$e_2 e_3=(\af+3)e_5$}\\

$\mathfrak{C}_{81}^5 $& $ : $ & 
$e_1 e_1 = e_2$& 
$e_1 e_2=e_3$&
$e_1e_3=  e_4$& 
$e_2e_2= 2 e_4$&
$e_2e_4=   e_5$ \\
		\end{longtable}}

All algebras from the present list are non-isomorphic, excepting		
\begin{longtable}{lcccccr}
$\mathfrak{C}_{13}^5(\alpha, \beta) \cong         
 \mathfrak{C}_{13}^5(\alpha,-\beta)$ & \ & $\mathfrak{C}^5_{26}(\alpha,\beta) \cong  \mathfrak{C}^5_{26}(\beta, \alpha)$ & \ & $\mathfrak{C}^5_{27}(\alpha) \cong  \mathfrak{C}^5_{27}(\frac 1 {\alpha})$ & \ &
${\mathfrak{C}}_{69}^{5}(\af) \cong  {\mathfrak{C}}_{69}^{5}(\sqrt[3]{1}\af)$ 
\end{longtable}
  \end{theorem}

\section{The geometric classification of complex $5$-dimensional nilpotent commutative $\mathfrak{CD}$-algebras}

\subsection{Degenerations of algebras}
Given an $n$-dimensional vector space ${\bf V}$, the set ${\rm Hom}({\bf V} \otimes {\bf V},{\bf V}) \cong {\bf V}^* \otimes {\bf V}^* \otimes {\bf V}$ 
is a vector space of dimension $n^3$. This space inherits the structure of the affine variety $\mathbb{C}^{n^3}.$ 
Indeed, let us fix a basis $e_1,\dots,e_n$ of ${\bf V}$. Then any $\mu\in {\rm Hom}({\bf V} \otimes {\bf V},{\bf V})$ is determined by $n^3$ structure constants $c_{i,j}^k\in\mathbb{C}$ such that
$\mu(e_i\otimes e_j)=\sum_{k=1}^nc_{i,j}^ke_k$. A subset of ${\rm Hom}({\bf V} \otimes {\bf V},{\bf V})$ is {\it Zariski-closed} if it can be defined by a set of polynomial equations in the variables $c_{i,j}^k$ ($1\le i,j,k\le n$).

The general linear group ${\rm GL}({\bf V})$ acts by conjugation on the variety ${\rm Hom}({\bf V} \otimes {\bf V},{\bf V})$ of all algebra structures on ${\bf V}$:
$$ (g * \mu )(x\otimes y) = g\mu(g^{-1}x\otimes g^{-1}y),$$ 
for $x,y\in {\bf V}$, $\mu\in {\rm Hom}({\bf V} \otimes {\bf V},{\bf V})$ and $g\in {\rm GL}({\bf V})$. Clearly, the ${\rm GL}({\bf V})$-orbits correspond to the isomorphism classes of algebras structures on ${\bf V}$. Let $T$ be a set of polynomial identities which is invariant under isomorphism. Then the subset $\mathbb{L}(T)\subset {\rm Hom}({\bf V} \otimes {\bf V},{\bf V})$ of the algebra structures on ${\bf V}$ which satisfy the identities in $T$ is ${\rm GL}({\bf V})$-invariant and Zariski-closed. It follows that $\mathbb{L}(T)$ decomposes into ${\rm GL}({\bf V})$-orbits. The ${\rm GL}({\bf V})$-orbit of $\mu\in\mathbb{L}(T)$ is denoted by $O(\mu)$ and its Zariski closure by $\overline{O(\mu)}$.

Let ${\bf A}$ and ${\bf B}$ be two $n$-dimensional algebras satisfying the identities from $T$ and $\mu,\lambda \in \mathbb{L}(T)$ represent ${\bf A}$ and ${\bf B}$ respectively.
We say that ${\bf A}$ {\it degenerates} to ${\bf B}$ and write ${\bf A}\to {\bf B}$ if $\lambda\in\overline{O(\mu)}$.
Note that in this case we have $\overline{O(\lambda)}\subset\overline{O(\mu)}$. Hence, the definition of a degeneration does not depend on the choice of $\mu$ and $\lambda$. If ${\bf A}\to {\bf B}$ and ${\bf A}\not\cong {\bf B}$, then ${\bf A}\to {\bf B}$ is called a {\it proper degeneration}. We write ${\bf A}\not\to {\bf B}$ if $\lambda\not\in\overline{O(\mu)}$ and call this a {\it non-degeneration}. Observe that the dimension of the subvariety $\overline{O(\mu)}$ equals $n^2-\dim\der({\bf A})$. Thus if ${\bf A}\to {\bf B}$ is a proper degeneration, then we must have $\dim\der({\bf A})>\dim\der({\bf B})$.

Let ${\bf A}$ be represented by $\mu\in\mathbb{L}(T)$. Then  ${\bf A}$ is  {\it rigid} in $\mathbb{L}(T)$ if $O(\mu)$ is an open subset of $\mathbb{L}(T)$.
Recall that a subset of a variety is called {\it irreducible} if it cannot be represented as a union of two non-trivial closed subsets. A maximal irreducible closed subset of a variety is called an {\it irreducible component}.
It is well known that any affine variety can be represented as a finite union of its irreducible components in a unique way.
The algebra ${\bf A}$ is rigid in $\mathbb{L}(T)$ if and only if $\overline{O(\mu)}$ is an irreducible component of $\mathbb{L}(T)$.

In the present work we use the methods applied to Lie algebras in \cite{GRH,GRH2}.
To prove 
degenerations, we will construct families of matrices parametrized by $t$. Namely, let ${\bf A}$ and ${\bf B}$ be two algebras represented by the structures $\mu$ and $\lambda$ from $\mathbb{L}(T)$, respectively. Let $e_1,\dots, e_n$ be a basis of ${\bf V}$ and $c_{i,j}^k$ ($1\le i,j,k\le n$) be the structure constants of $\lambda$ in this basis. If there exist $a_i^j(t)\in\mathbb{C}$ ($1\le i,j\le n$, $t\in\mathbb{C}^*$) such that the elements $E_i^t=\sum_{j=1}^na_i^j(t)e_j$ ($1\le i\le n$) form a basis of ${\bf V}$ for any $t\in\mathbb{C}^*$, and the structure constants $c_{i,j}^k(t)$ of $\mu$ in the basis $E_1^t,\dots, E_n^t$ satisfy $\lim\limits_{t\to 0}c_{i,j}^k(t)=c_{i,j}^k$, then ${\bf A}\to {\bf B}$. In this case  $E_1^t,\dots, E_n^t$ is called a {\it parametric basis} for ${\bf A}\to {\bf B}$.

To prove a non-degeneration ${\bf A}\not\to {\bf B}$ we will use the following lemma (see \cite{GRH}).

\begin{lemma}\label{main}
Let $\mathcal{B}$ be a Borel subgroup of ${\rm GL}({\bf V})$ and $\mathcal{R}\subset \mathbb{L}(T)$ be a $\mathcal{B}$-stable closed subset.
If ${\bf A} \to {\bf B}$ and ${\bf A}$ can be represented by $\mu\in\mathcal{R}$ then there is $\lambda\in \mathcal{R}$ that represents ${\bf B}$.
\end{lemma}

In particular, it follows from Lemma \ref{main} that ${\bf A}\not\to {\bf B}$, whenever $\dim({\bf A}^2)<\dim({\bf B}^2)$.

When the number of orbits under the action of ${\rm GL}({\bf V})$ on  $\mathbb{L}(T)$ is finite, the graph of primary degenerations gives the whole picture. In particular, the description of rigid algebras and irreducible components can be easily obtained. Since the variety of $5$-dimensional nilpotent commutative $\mathfrak{CD}$-algebras contains infinitely many non-isomorphic algebras, we have to fulfill some additional work. Let ${\bf A}(*):=\{{\bf A}(\alpha)\}_{\alpha\in I}$ be a family of algebras and ${\bf B}$ be another algebra. Suppose that, for $\alpha\in I$, ${\bf A}(\alpha)$ is represented by a structure $\mu(\alpha)\in\mathbb{L}(T)$ and ${\bf B}$ is represented by a structure $\lambda\in\mathbb{L}(T)$. Then by ${\bf A}(*)\to {\bf B}$ we mean $\lambda\in\overline{\cup\{O(\mu(\alpha))\}_{\alpha\in I}}$, and by ${\bf A}(*)\not\to {\bf B}$ we mean $\lambda\not\in\overline{\cup\{O(\mu(\alpha))\}_{\alpha\in I}}$.

Let ${\bf A}(*)$, ${\bf B}$, $\mu(\alpha)$ ($\alpha\in I$) and $\lambda$ be as above. To prove ${\bf A}(*)\to {\bf B}$, it is enough to construct a family of pairs $(f(t), g(t))$ parametrized by $t\in\mathbb{C}^*$, where $f(t)\in I$ and $g(t)=\left(a_i^j(t)\right)_{i,j}\in {\rm GL}({\bf V})$. Namely, let $e_1,\dots, e_n$ be a basis of ${\bf V}$ and $c_{i,j}^k$ ($1\le i,j,k\le n$) be the structure constants of $\lambda$ in this basis. If we construct $a_i^j:\mathbb{C}^*\to \mathbb{C}$ ($1\le i,j\le n$) and $f: \mathbb{C}^* \to I$ such that $E_i^t=\sum_{j=1}^na_i^j(t)e_j$ ($1\le i\le n$) form a basis of ${\bf V}$ for any  $t\in\mathbb{C}^*$, and the structure constants $c_{i,j}^k(t)$ of $\mu\big(f(t)\big)$ in the basis $E_1^t,\dots, E_n^t$ satisfy $\lim\limits_{t\to 0}c_{i,j}^k(t)=c_{i,j}^k$, then ${\bf A}(*)\to {\bf B}$. In this case, $E_1^t,\dots, E_n^t$ and $f(t)$ are called a {\it parametric basis} and a {\it parametric index} for ${\bf A}(*)\to {\bf B}$, respectively. In the construction of degenerations of this sort, we will write $\mu\big(f(t)\big)\to \lambda$, emphasizing that we are proving the assertion $\mu(*)\to\lambda$ using the parametric index $f(t)$.



\subsection{The geometric classification of $5$-dimensional nilpotent 
commutative $\mathfrak{CD}$-algebras}
The geometric classification of $5$-dimensional nilpotent 
commutative $\mathfrak{CD}$-algebras is based on some previous works:
namely, all irreducible components of $5$-dimensional nilpotent associative commutative algebras are given in \cite{maz79} and all degenerations between these algebras are
given in \cite{klp19};
  all irreducible components of $5$-dimensional nilpotent Jordan algebras were described in \cite{kama}.
In the proof of the present theorem we give all necessary arguments for the description of all irreducible components of the variety of
$5$-dimensional nilpotent  commutative $\mathfrak{CD}$-algebras.

\begin{theorem}\label{main-geo}
The variety of complex $5$-dimensional nilpotent commutative $\mathfrak{CD}$-algebras  is $24$-dimensional and it has $10$ irreducible components.
In particular, there are $6$ rigid algebras:  
non-Jordan algebras $\cc^5_{69}, \cc^5_{72}, \cc^5_{76}, \cc^5_{77}, \cc^5_{81}$
and Jordan algebra ${\mathcal J}_{21}.$
\end{theorem}

\begin{proof}[{\bf Proof}]
Thanks to \cite{kama} the algebras 
 $\epsilon_1,$ ${\mathcal J}_{21},$ ${\mathcal J}_{22},$
 ${\mathcal J}_{27}(\varepsilon,\phi)$ and ${\mathcal J}_{40}$
 determine the irreducible components in the variety of complex 
 $5$-dimensional nilpotent Jordan algebras 
 (which is a proper subvariety of nilpotent commutative $\mathfrak{CD}$-algebras), 
 where
\begin{longtable}{lllllllllll}
$\epsilon_1$  &$ : $& $e_1e_1= e_2$ & $e_2e_2=e_4$& $e_1e_3=e_4$& $e_1e_2=e_3$ & $e_1e_4=e_5$ & $e_2e_3=e_5$\\

${\mathcal J}_{21}$ &$ : $ &  $e_1e_1=e_5$ & $e_1e_2=e_4$ & $e_2e_2=e_5$ & $e_3e_3=e_4$ & $e_3e_4=e_5$  \\

${\mathcal J}_{22}$ &$ : $ &  $e_1e_1=e_2$ & $e_1e_2=e_4$ & $e_1e_4=e_5$ & $e_2e_2=e_5$ & $e_3e_3=e_4$ \\

${\mathcal J}_{27}(\varepsilon,\phi)$ &$ : $ &  $e_1e_1=e_3$ & $e_1e_3=\phi e_5$& $e_1e_4=e_5$ & $e_2e_2=e_4$ & $e_2e_3=e_5$ & $e_2e_4=\varepsilon e_5$ \\

${\mathcal J}_{40}$ &$ : $ & $e_1e_1=e_5$ & $e_1e_2=e_3$ & $e_1e_3=e_4$ & $e_2e_2=e_4$ & $e_2e_3=e_5$

        \end{longtable}

Let us give the list of usefull degenerations:
{\tiny       
\begin{longtable}{l ll}
 \hline
 $\cc^5_{78}  \to \epsilon_1$ & 
 $E_1^t= t^2  e_1 $ & $E_2^t=t^4  e_2 $ \\ 
 $E_3^t= -t^8 e_5 $ & $E_4^t= i t^5 e_4$ & $E_5^t= t^4 e_3$\\

   \hline
 $\cc^5_{16}(t^{-2})  \to {\mathcal J}_{22}$ & 
\multicolumn{2}{l}{ $E_1^t= t^2( 1+ t^2) e_1 - \frac{t^6 (1 + t^2)^2}{3 + 2 t^2} e_2 - t^4 (1 + t^2) e_3 - \frac{t^{10} (1 + t^2)^2}{2 (3 + 2 t^2)^2} e_4 $}\\ 
& 
\multicolumn{2}{l}{  $E_2^t= t^4 (1 + t^2)^2 e_2 + t^4 (1 + t^2)  e_3 + \frac{t^8 (1 + t^2)^2}{3 + 2 t^2} e_4 $} \\ 
 $E_3^t= t^3 (1 + t^2)  e_3 $ & 
 $E_4^t= t^6 (1 + t^2)^2 e_4 $ & 
 $E_5^t= t^6 (1 + t^2)^4 e_5$\\

  \hline
 $\cc^5_{52}\left(\frac{\phi}{\sqrt{t}}\right)  \to {\mathcal J}_{27}(\varepsilon,\phi)$ & 
 $E_1^t= t^{\frac{5}{2}} e_1 - \varepsilon t^4 e_3 + t^3 (1 + \varepsilon t) e_4 $ \\& 
 $E_2^t= t^2 e_2 + \varepsilon t^3 e_4 $ & 
$E_3^t= t^5 e_3 - t^5 e_4 + t^6 (1 - 2 \varepsilon \phi - \varepsilon^2 t) e_5  $  \\& 
 $E_4^t= t^4 e_4 + \varepsilon^2 t^6 e_5$ & 
 $E_5^t= t^7 e_5$\\
 
  \hline
 $\cc^5_{26}(4 t^3, -3)  \to {\mathcal J}_{40}$ & 
 $E_1^t=-2 t e_1  $ & 
 $E_2^t= -2 t^2 e_2 + 2 t^2 e_3$ \\ 
 $E_3^t= 4 t^3 e_3 - 4 t^3 e_4 - 4 t^3 e_5 $ & 
 $E_4^t=-8 t^4 e_4 - 8 t^4 e_5$ & 
 $E_5^t= 16 t^5 e_5 $\\

 \hline
 $\cc^5_{13}(-1,0) \to \cc^5_{01}$ & 
 $E_1^t= t^{-1}  e_1 $ & $E_2^t=t^{-2}  e_2 $ \\ 
 $E_3^t= t^{-3} e_3+t^{-3} e_4+t^{-3} e_5 $ & $E_4^t= t^{-4} e_4+t^{-4} e_5$ & $E_5^t= t^{-5} e_5$\\

 \hline
 $\cc^5_{13}(-1-\A, 0) \to \cc^5_{02}(\A)$ & 
 $E_1^t=  e_1 - t e_4$ & $E_2^t= e_2 + t ^2 e_5  $ \\ 
 $E_3^t= e_3 $ & $E_4^t=- e_5 $ & $E_5^t= t e_4$\\

 \hline
 $\cc^5_{13}\left(\frac{1}{t-1},-\frac{\sqrt{t^3}}{2 \sqrt{t-1}} \right) \to \cc^5_{03}$ & 
 $E_1^t= \sqrt[4]{t-1} e_1$ &  $E_2^t=\sqrt{t - 1} e_2 + \sqrt{t^3}  e_4 $\\ 
 $E_3^t= \frac{\sqrt[4]{(t - 1)^3}}{t} e_3 $  & $E_4^t= e_5 $ & $E_5^t= \sqrt{t} e_4$\\

 \hline
 
 $\cc^5_{13}(-1,0 ) \to \cc^5_{04}$ &  
 $E_1^t= t e_1 $ & $E_2^t=t^2 e_2 $ \\ 
 $E_3^t= t e_3 + i (t^2) e_4   $ & $E_4^t= -t^4 e_5  $ & $E_5^t= i t^3 e_4$\\
 
  \hline
 $\cc^5_{26}(0,0) \to \cc^5_{05}$ &  $E_1^t=t  e_1 $ & $E_2^t= t^{-1} e_2 $ \\ 
 $E_3^t=e_3   $ & $E_4^t=  e_5  $ & $E_5^t= t^{-2} e_4$\\
 
   \hline
 $\cc^5_{26}(t^{-2}, t^{-2}) \to \cc^5_{06}$ & 
 $E_1^t= t^{-1} e_1$& $E_2^t=  t^{-1} e_2$ \\ 
 $E_3^t=t^{-2} e_3   $ & $E_4^t= t^{-4}  e_5  $ & $E_5^t=  t^{-4} e_4$\\
 
   \hline
 $\cc^5_{26}(4 t^{-2},0  ) \to \cc^5_{07}$ & $E_1^t= t^{3} e_1$ &  $E_2^t=  t^{-1} e_2 $ \\ 
 $E_3^t=t^2 e_3   $ & $E_4^t= t^4  e_5  $ & $E_5^t=   e_4 $\\

    \hline
 $\cc^5_{26}(0,1) \to \cc^5_{08}$ &  $E_1^t= t  e_1 - t e_2 - t e_3$ & $E_2^t= -2 t^2 e_3 $\\ 
 $E_3^t= -t^2 e_2 - t^2 e_3   $ & $E_4^t= t^3 e_4 + t^3 e_5  $ & $E_5^t=5 t^4 e_5 $\\

    \hline
 $\cc^5_{26}(5,4) \to \cc^5_{09}$ & $E_1^t=-t e_1 + t e_2 + t e_3    $ &  $E_2^t=-2 t^2 e_3 + 8 t^2 e_5$\\ 
 $E_3^t=t^2 e_2 $&  $E_4^t= t^3 e_4  $ & $E_5^t=4 t^4 e_5 $\\

 \hline
  $\cc^5_{26}(-1,0) \to \cc^5_{10}$ &$E_1^t=-t e_1 + t^{-1} e_2 + t e_3   $&$E_2^t= -2 e_3 + 2 e_4$\\ $E_3^t= e_2 $ & $E_4^t=(   2 t-2t^{-1}) e_4  $ & $E_5^t=4 e_5  $\\

 \hline
   $\cc^5_{13}\left(\frac{1}{t-1},0 \right) \to \cc^5_{11}$ &   
 $E_1^t=  e_1$ & $E_2^t=   e_2$\\ 
 $E_3^t= t^{-1} e_3 - \frac{\sqrt{t}}{\sqrt{t-1}} e_4  $  &$E_4^t= \frac{\sqrt{t^3}}{\sqrt{t-1}} e_4  $ & $E_5^t= \frac{1}{t-1}e_5$\\

 \hline
    $\cc^5_{13}(\A,t^{-\frac{1}{2}} ) \to \cc^5_{12}(\A)$ &  $E_1^t=  e_1$ & $E_2^t=   e_2$\\  
 $E_3^t=e_3 + t^{\frac{3}{2}} e_4  $ & $E_4^t= \sqrt{t} e_4  $ & $E_5^t= e_5$\\

   \hline
    $\cc^5_{80}\left(\frac{3(1+\A)}{\B\sqrt{\B^2-\A}-1-\B^2}\right) \to \cc^5_{13}(\A,\B)$ & 
\multicolumn{2}{l}{ $E_1^t= te_1+
\frac{3 t (1+\A) \left(2-\B^2+\B \sqrt{\B^2-\A}\right)}{4 \left(4+3\A+(2\A +5 )\B^2+2 \B^4-\left(5 +3\A +2 \B^2 \right)\B \sqrt{\B^2-\A}\right)}e_2$}\\
& \multicolumn{2}{l}{$+\frac{27 t \left(-2+\B^2-\B \sqrt{\B^2-\A}\right)^3 \left(1+\A\right)^2\sqrt{\B^2-\A} \left((2+4 \A) \B-2 \B^3-\left(2+3 \A-2 \B^2 \right)\sqrt{\B^2-\A}\right)}{128 \left(1+\B^2-\B\sqrt{\B^2-\A}\right)^4 \left(4+3 \A+\B^2-\B \sqrt{\B^2-\A}\right)^3}e_3$}\\
& \multicolumn{2}{l}{$-\frac{3 t \left(2-\B^2+\B \sqrt{\B^2-\A}\right)^2 (1+\A) \left(2+3 \A-\B^2+\B \sqrt{\B^2-\A}\right)\left(\sqrt{\B^2-\A}-\B\right)}{32 \left(4+3\A+(2\A +5 )\B^2+2 \B^4-\left(5 +3\A +2 \B^2 \right)\B \sqrt{\B^2-\A}\right)^2}e_4$}\\

& \multicolumn{2}{l}{$E_2^t=   t^2e_2-\frac{3 t^2\left(1+\A\right) \left(\sqrt{\B^2-\A}-\B\right)\left(-2+\B^2-\B\sqrt{\B^2-\A}\right)}{2 \left(1+\B^2-\B\sqrt{\B^2-\A}\right) \left(4+3 \A+B^2-\B \sqrt{\B^2-\A}\right)}e_4 $}\\
& \multicolumn{2}{l}{$E_3^t= \frac{9 t^2 (1+\A) \left(2 \B \left(5+\B^2\right) \left(-\B+\sqrt{\B^2-\A}\right)+\A \left(2-6 \B^2+7 \B \sqrt{\B^2-\A}\right)\right)}{4 \left(1+\B^2-\B \sqrt{\B^2-\A}\right)^2 \left(4+3 \A+\B^2-\B \sqrt{\B^2-\A}\right)}e_3+ $} \\
& \multicolumn{2}{l}{$t^2(\sqrt{\B^2-\A}-\B)e_4+\frac{9 t^2 \left(1+\A \right){\bf Q}}{32 \left(1+\B^2-\B \sqrt{-\A+\B}\right)^3 \left(4+3 \A+\B^2-\B \sqrt{\B^2-\A}\right)^2}e_5$}\\

&  \multicolumn{2}{l}{$E_4^t= -\frac{3 t^3 \left(1+\A\right)}{1+\B^2-\B\sqrt{\B^2-\A}}e_3+\frac{9 t^3 \left(\B-\sqrt{\B^2-\A}\right) (1+\A) \left(2 (1+\A) \B-\sqrt{\B^2-\A}\right)}{\left(1+\B^2-\B\sqrt{\B^2-\A}\right)^2 \left(4+3\A+\B^2-\B \sqrt{\B^2-\A}\right)}e_5  $} \\ 
&\multicolumn{2}{l}{$E_5^t=-\frac{3 t^4}{1+\B^2-\B\sqrt{\B^2-\A}}e_5$}\\
\multicolumn{3}{l}{${\bf Q} = -88\A-60\A^2+344 \B^2-204\A \B^2-558\A^2 \B^2-201\A^3 \B^2+1032 \B^4+966\A \B^4+233\A^2 \B^4+360 \B^6+232\A \B^6- 8\B^8 $}\\
\multicolumn{3}{l}{$ \sqrt{\B^2-\A} \left(-344 \B-312\A \B-60\A^2 \B-1032 \B^3-1146\A \B^3-346\A^2 \B^3-360 \B^5-228\A \B^5+8 \B^7\right)$}\\
 \hline
 
  $\cc^5_{26}((t^2-1)^2t^{-4}, t^{-4}) \to \cc^5_{14}$ & 
  $E_1^t= - e_1 +  e_2 +  e_3 $& 
  $E_2^t= -2 e_3 + 2(1 -  t^2)t^{-4} e_5 $\\ 
 $E_3^t= t e_2 + t^{-1} e_3 $ &$E_4^t=2 e_4 + 2t^{-2}     e_5  $ & $E_5^t= 4   e_5   $\\

 \hline
  $\cc^5_{26}(1+t^{-4}-\frac{2}{ t^2}-8 t, t^{-4}) \to \cc^5_{15}$ &      
   $E_1^t= t  e_1 - t e_2 - t e_3 $ & 
    $E_2^t= -2 t^2 e_3 - (2 - 2t^{-2} + 8 t^3) e_5 $\\   
  $E_3^t= -t^2 e_2 - e_3 $ 
  & $E_4^t= 2 t^2 e_4 + 2 e_5  $& $E_5^t= 4 t^4 e_5   $\\

 \hline
  $\cc^5_{26}\left(\frac{t^4-8 t^7+16 t^8-8 t^9}{(1-2 t)^2}, \frac{(t-1)^4}{(1-2 t)^2} \right) \to \cc^5_{17}$ & 
  \multicolumn{2}{l}{     $E_1^t= -\frac{(t-1)^2 t^3}{2 t-1} e_1 + \frac{(t-1 ) t^4}{2t-1}    e_2 + \frac{(t-1 )^2 t^4}{(1 - 2 t)^2} e_3  $} \\&
\multicolumn{2}{l}{ $E_2^t=\frac{2 ( 1- t)^3 t^7}{(1 - 2 t)^2} e_3 + \frac{2 (1 - t)^3 t^7}{(1 - 2 t)^3} e_4 - \frac{  2 (t-1 )^6 t^7 (4 t^6-1 - t )}{(1 - 2 t)^4} e_5 $}\\ 
   $E_3^t= \frac{(1 - t) t^5}{1 - 2 t} e_2 - \frac{(t-1)^3 t^5}{(1 - 2 t)^2} e_3 $  &
  $E_4^t= \frac{2 (t-1 )^4 t^{10}}{(1 - 2 t)^3} e_4 +  \frac{ 2 (1 - t)^6 t^{10}}{(1 - 2 t)^4}e_5 $& 
 $E_5^t= \frac{(4 (1 - t)^6 t^{14}}{(1 - 2 t)^4} e_5   $\\

 \hline
  $\cc^5_{26}(1,0) \to \cc^5_{18}$ & 
     $E_1^t=t^{-1} e_1 - t^{-1} e_2 - t^{-1} e_3   $ &  $E_2^t=-2t^{-2} e_3$\\
     $E_3^t=- e_2 $ &   $E_4^t=2t^{-2} e_4 $ & $E_5^t=4t^{-4} e_5  $\\

 \hline
  $\cc^5_{26} ( 1+4t^{-4},4t^{-4}) \to \cc^5_{19}$ & 
     $E_1^t=t^{-1} e_1 - t^{-1} e_2 - t^{-1} e_3   $ &
 $E_2^t=-2t^{-2} e_3 + 8t^{-6} e_5 $\\  
 $E_3^t=- e_2 $ & 
 $E_4^t=2t^{-2} e_4 $ & $E_5^t=4t^{-4} e_5  $\\

 \hline
  $\cc^5_{26} (1+8t^{-3},0 ) \to \cc^5_{20}$ & 
     $E_1^t=t^{-1} e_1 - t^{-1} e_2 - t^{-1} e_3   $&
 $E_2^t=-2t^{-2} e_3 + 8t^{-5} e_5 $\\ 
 $E_3^t=- e_2 $ &   $E_4^t=2t^{-2} e_4 $ & $E_5^t=4t^{-4} e_5  $\\

 \hline
  $\cc^5_{26} \left(\frac{4-8 t+t^2+2 t^4+t^6}{t^6},\frac{4+t^2}{t^6} \right) \to \cc^5_{21}$ & 
$E_1^t= t^{-2} e_1 - t^{-2} e_2 - t^{-2} e_3   $ \\&
$E_2^t=-2t^{-4} e_3 + 2t^{-4} e_4  + 2 (t-2)^2t^{-10} e_5$ & 
$E_3^t=-t^{-1} e_2 + t^{-3} e_3 $  \\&
  $E_4^t= 2 t^{-5} e_4 - 2 t^{-7}e_5 $& $E_5^t=4 t^{-8} e_5  $\\
 
  \hline
  $\cc^5_{26}(0,0) \to \cc^5_{22}$ & 
     $E_1^t= -\frac{1}{2} e_1 + \frac{1}{2} e_2    $& $E_2^t=-\frac{1}{2} e_3$\\ 
     $E_3^t= \frac{t}{2}e_2$  &  $E_4^t=  -\frac{t}{4} e_4$ & $E_5^t= \frac{1}{4} e_5  $\\
 
  \hline
  $\cc^5_{26}(t^{-2},t^{-2}) \to \cc^5_{23}$ & 
$E_1^t= -\frac{1}{2} e_1 + \frac{1}{2} e_2    $& $E_2^t=-\frac{1}{2} e_3 + \frac{1}{2 t^2} e_5$\\ $E_3^t= -\frac{t}{2}e_2$ &  $E_4^t=  \frac{t}{4} e_4$ & $E_5^t= \frac{1}{4} e_5  $\\
 
  \hline
$\cc^5_{26}((\A+2 t)t^{-2},\A t^{-2}) \to \cc^5_{24}(\A)$ & 
$E_1^t= -\frac{1}{2} e_1 + \frac{1}{2} e_2$& $E_2^t= -\frac{1}{2} e_3 + \frac{\A + t}{2 t^2} e_5$\\ $E_3^t= -\frac{t}{2}e_2$  & $E_4^t=  \frac{t}{4} e_4$ & $E_5^t= \frac{1}{4} e_5  $\\

  \hline
$\cc^5_{80} \left( \frac{1}{t-1}\right) \to \cc^5_{25}$ & 
$E_1^t= t e_1 + \frac{t}{8 - 4 t}e_2$\\& 
$E_2^t= t^2  e_2 + \frac{t^2}{4 - 2 t} e_3 + \frac{t^3}{16 (t-2)^2 (t-1)} e_4$ &
$E_3^t= t^3 e_3 - \frac{t^3 (2 + t)}{4 (t-2) (t-1)} e_4 + \frac{ t^3 (7 t-4)}{16 (t-2)^2 (t-1)} e_5$\\ 
& 
$E_4^t= \frac{t^4}{t-1} e_4 - \frac{t^5}{2 - 3 t + t^2} e_5 $ & 
$E_5^t= t^4 e_5  $\\

   \hline
$\cc^5_{26}(1+\A t^{-2},t^{-2}) \to \cc^5_{27}(\A)$ & 
$E_1^t= t^{-1} e_1 - t^{-1} e_3    $& $E_2^t= t^{-1} e_2 $\\ 
$E_3^t= t^{-2} e_3 - t^{-2} e_4$ &  $E_4^t=  t^{-3} e_4$ & $E_5^t= t^{-4}  e_5  $\\

   \hline
$\cc^5_{26}(1+4 t^{-2},0 ) \to \cc^5_{28}$ & 
$E_1^t= ( 2t^{-1}-1) e_1 +  e_2 + (t-2)t^{-1} e_3    $& 
$E_2^t= 2 t^{-1}  e_2 $\\ 
$E_3^t=(4 - 2 t)t^{-2} e_3 + 2 (t-2 )t^{-2} e_4 $ & 
$E_4^t=  (8 - 4 t)t^{-3} e_4  $ & $E_5^t=  4 (t-2)^2 t^{-4} e_5  $\\

    \hline
  $\cc^5_{26}(1,0) \to \cc^5_{29}$ & $E_1^t= t^{-1} e_1 - t^{-1} e_3   $& $E_2^t= t^{-1}  e_2 $\\ $E_3^t= t^{-2} e_3 - t^{-2} e_4 $ &  $E_4^t= t^{-3}  e_4  $ & $E_5^t=  t^{-4} e_5   $\\

     \hline
  $\cc^5_{26} \left(1+\A t^2+\frac{t^4}{4},\frac{t^4}{4}  \right) \to \cc^5_{30}(\A)$ & 
$E_1^t=-\frac{t^2}{2} e_1 +\frac{ t^2}{2} e_2 + \frac{t^2}{2} e_3   $&
 $E_2^t= t e_2 $\\ 
 $E_3^t=-\frac{t^3}{2} e_3 + \frac{t^3}{2} e_4 + \frac{t^7}{8} e_5 $ & 
 $E_4^t=-\frac{t^4}{2} e_4      $ & $E_5^t= \frac{t^6}{4} e_5   $\\

      \hline
$\cc^5_{26}(1+4 t^2,0) \to \cc^5_{31}$ & 
$E_1^t= -2 t^2 e_1 + 2 t^2 e_2 + 2 t^2  e_3  $&
 $E_2^t=2 t  e_2  $\\ 
 $E_3^t= -4 t^3 e_3 + 4 t^3 e_4 $ & 
 $E_4^t= -8 t^4 e_4  $ & $E_5^t= 16 t^6 e_5   $\\

       \hline
  $\cc^5_{26}(1,0) \to \cc^5_{32}$ & 
$E_1^t=-\frac{t^2}{2} e_1 + \frac{t^2}{2} e_2 + \frac{t^2}{2} e_3  $&
 $E_2^t=2 t  e_2  $\\ 
 $E_3^t= -\frac{t^3}{2} e_3 +\frac{ t^3}{2} e_4 $ & 
 $E_4^t= -\frac{t^4}{2} e_4  $ & $E_5^t= \frac{t^6}{4} e_5   $\\

        \hline
$\cc^5_{26}(2+t^2,t^2) \to \cc^5_{33}$ & 
$E_1^t=t  e_1 - t e_2 - t e_3  $&
 $E_2^t=  e_2 $\\ 
 $E_3^t=t e_3 - t  e_4 - t^3 e_5 $ & 
 $E_4^t= t e_4   $ & $E_5^t=t^2 e_5    $\\

        \hline
  $\cc^5_{26}(2,0) \to \cc^5_{34}$ & 
$E_1^t=t  e_1 - t e_2 - t e_3  $&

 $E_2^t=  e_2 $\\ 
 $E_3^t=t e_3 - t e_4 $ & 
 $E_4^t= t e_4   $ & $E_5^t=t^2 e_5    $\\

        \hline
  $\cc^5_{26}(2,1) \to \cc^5_{35}$ & 
$E_1^t= e_1 - e_2 - e_3  $&
$E_2^t= t^{-1} e_2 $\\ 
$E_3^t=t^{-1} e_3 - t^{-1} e_4-t^{-1}e_5 $ & 
 $E_4^t= t^{-2} e_4   $ & $E_5^t=t^{-2} e_5    $\\
 
     \hline
  $\cc^5_{26}(1,0) \to \cc^5_{36}$ &   
     $E_1^t= t e_1 - t e_2 - t e_3  $&
 $E_2^t=   e_2 $\\
$E_3^t=t e_3 - t  e_4 $ & 
 $E_4^t= t e_4   $ & $E_5^t=t^{2} e_5    $\\

     \hline
  $\cc^5_{26}(1,0) \to \cc^5_{37}$ & $E_1^t=-2 t^3 e_1 + 2 t^3 e_2 + 2 t^3 e_3  $&
 $E_2^t=  -2 t^2 e_2 + 2 t^4 e_3  $\\ 
 $E_3^t= 4 t^5 e_3 - 4 t^5 e_4 - 4 t^9 e_5 $ & 
 $E_4^t=-8 t^6 e_4 + 8 t^8 e_5   $ & $E_5^t= 16 t^{10} e_5    $\\

      \hline
  $\cc^5_{26}((1+t^2)^2 , t^4 ) \to \cc^5_{38}$ & 
$E_1^t= -2 t^3 e_1 + 2 t^3 e_2 + 2 t^3 e_3  $&
$E_2^t=  -2 t^2 e_2 + 2 t^4  e_3  $\\ 
$E_3^t=4 t^5 e_3 - 4 t^5 e_4 - 4 t^9 e_5 $ & 
 $E_4^t = -8 t^6 e_4 + 8 t^8 e_5   $ & $E_5^t= 16 t^{10} e_5    $\\
 
       \hline
$\cc^5_{26}(1,1) \to \cc^5_{39}$ & 
$E_1^t= -t e_1 + t  e_2 + t  e_3   $& $E_2^t=   e_2 +   e_3   $\\ 
$E_3^t= -t e_3 + t e_4 + t  e_5$ & $E_4^t =2 e_4 + 2 e_5   $ & $E_5^t= t^2 e_5    $\\

        \hline
$\cc^5_{26}(0,1) \to \cc^5_{40}$ & 
$E_1^t= -t e_1 + t  e_2 + t  e_3   $& $E_2^t=   e_2 +   e_3   $\\ 
$E_3^t= -t e_3 + t e_4 + t  e_5$ & 
 $E_4^t =2 e_4 + 2 e_5   $ & $E_5^t= t^2 e_5    $\\

        \hline
$\cc^5_{56}(1+t^3) \to \cc^5_{41}$ & 
$E_1^t= t e_1 $& $E_2^t= t^2 e_4 + t e_5    $\\ 
$E_3^t= -t e_3 + t  e_4 + e_5$ & 
$E_4^t =t^2  e_2 - e_3 +  e_4   $ & $E_5^t=t^4 e_5    $\\

     \hline
$\cc^5_{56}(2) \to \cc^5_{42}$ & 
$E_1^t=\sqrt{t} e_1 + i \sqrt{t} e_2   $& 
$E_2^t= t e_3 - t  e_4    $\\ 
$E_3^t=i t e_2$ & 
$E_4^t = -t e_4 - t^{\frac{3}{2}} e_5   $ & $E_5^t= t^2 e_5    $\\

     \hline
$\cc^5_{52}(i ) \to \cc^5_{43}$ & 
$E_1^t= it^{-2} e_1 + t^{-2} e_2     $& 
$E_2^t=  -t^{-1} e_2 + \frac{1}{2t^3} e_3 - \frac{1 + 2 t }{2 t^3} e_4   $\\ 
$E_3^t= -t^{-3} e_4  $ & 
 $E_4^t = -t^{-5} e_3 + t^{-5} e_4 - t^{-7} e_5 $ & 
 $E_5^t=  t^{-6} e_5   $\\

    \hline
$\cc^5_{49}(1+t^4 ) \to \cc^5_{44}$ & 
$E_1^t=-(-1)^{\frac{3}{4}} e_1 + (-1)^{\frac{1}{4}} e_2   $& 
$E_2^t=  (-1)^{\frac{1}{4}} t  e_2   $\\ 
$E_3^t= i t e_4 + t e_5 $ & 
$E_4^t = it^{-1} e_3 - it^{-1} e_4 - (2 + t^2)t^{-1} e_5 $ & 
$E_5^t= -t^2 e_5  $\\

   \hline
$\cc^5_{56}\left(\frac{1+2 t^2-2 t^3+2 t^4-2 t^5}{(1+t^2-t^3)^2}\right) \to \cc^5_{45}$ & 
$E_1^t=  t \sqrt{t^3 - t^2 -1 } e_1 + t e_2    $& 
$E_2^t=  t^2 e_2    $\\  
$E_3^t= t^3 e_4$ & 
$E_4^t =  t (t^3 - t^2 -1 ) e_3 + (t + t^3)  e_4 - t^5 e_5   $ & $E_5^t= t^6 e_5    $\\

     \hline
$\cc^5_{52}(i (t-1) ) \to \cc^5_{46}$ & 
$E_1^t= (t^3-1)^{\frac{3}{2}}t^{-2} e_1 + (t^{-2} - t) e_2      $& 
$E_2^t=  (t^{-1} - t^2) e_2  $\\ 
$E_3^t=(-2 + t^{-3} + t^3) e_4$ & 
 $E_4^t =(1 - t^3)^3t^{-5} e_3 - (1 - t^3)^2)t^{-5} e_4   $ & $E_5^t= (1 - t^3)^4t^{-6} e_5   $\\

     \hline
$\cc^5_{52}\left(\frac{1}{\sqrt{t^3-1}}\right) \to \cc^5_{47}$ & 
$E_1^t= t^{-2} e_1 + it^{-2} e_2 - it^{-4} e_3 + i t^{-4} e_4        $
& 
$E_2^t=  it^{-1} e_2  - t^{-2} e_4    $ \\ 
$E_3^t=-t^{-3} e_4 + t^{-5} e_5$ & 
 $E_4^t =  -it^{-5} e_3 + it^{-5} e_4 - 2 i t^{-6} e_5  $ & $E_5^t= t^{-6} e_5  $\\

     \hline
$\cc^5_{56}(1+t^4) \to \cc^5_{48}$ & 
 $E_1^t= i t e_1 + t e_2        $& $E_2^t=  t^2 e_2    $ \\ 
$E_3^t=  t^3 e_4$ & 
 $E_4^t =    -t e_3 + t  e_4 $ & $E_5^t= t^6 e_5 $\\
 
  \hline
$\cc^5_{49}(1+t^{-4}) \to \cc^5_{50}$ & 
\multicolumn{2}{l}{ $E_1^t=(-1)^{\frac{3}{4}} t^{\frac{3}{2}} e_1 + (-1)^{\frac{1}{4}} t^{\frac{3}{2}} e_2 - i t^4 e_3 + 
 i t^4 e_4        $}\\
& $E_2^t= (-1)^{\frac{1}{4}} t^{\frac{1}{2}} e_2 - i t^4 e_3 + i t^4 e_4      $ &
$E_3^t= -i t^3 e_3 + i t^3 e_4 $ \\
& 
 $E_4^t =   i t e_4 - t^3 e_5 $ & $E_5^t= -t^2 e_5 $\\

     \hline
$\cc^5_{49}(0) \to \cc^5_{51}$ & 
 $E_1^t= -t^{-1} e_1        $& $E_2^t= -t e_2    $ \\ 
$E_3^t=  t^{-2} e_3$ & 
 $E_4^t =     t^2  e_4 $ & $E_5^t=  e_5 $\\
 
     \hline
$\cc^5_{49}(1+t) \to \cc^5_{52}(\A)$ & 
 $E_1^t= -t^{-1} e_1        $& $E_2^t= -t e_2    $ \\ 
$E_3^t=  t^{-2} e_3$ & 
 $E_4^t =     t^2  e_4 $ & $E_5^t=  e_5 $\\

        \hline
  $\cc^5_{52}(t^{-1}) \to \cc^5_{53}$ & 
\multicolumn{2}{l}{$E_1^t=-\sqrt{\frac{t}{\A - t}} e_1 + \sqrt{\frac{t^3}{\A - t}} e_2 + 
 \frac{\A}{\A - t} e_3 - \frac{\A}{\A - t} e_4    $}\\
 & 
 $E_2^t=  \sqrt{\frac{t}{\A - t}} e_2 + \frac{1}{\A - t} e_3 + \frac{1}{t-\A} e_4  $ & 
$E_3^t= \frac{t}{\A - t} e_3 + t e_5 $ \\ &  
$E_4^t = \frac{t}{\A - t} e_4 + \frac{t}{(\A - t)^2} e_5 $ & 
$E_5^t= \frac{t^2}{(\A - t)^2} e_5   $\\

        \hline
  $\cc^5_{52}(t^{-4}) \to \cc^5_{54}$ & $E_1^t=   e_1   $&$E_2^t=  t^{-1}  e_2  $\\ 
  $E_3^t=  e_3 $ & $E_4^t = t^{-2} e_4  $ & $E_5^t= t^{-4}  e_5   $\\
 
        \hline
$\cc^5_{52}(0) \to \cc^5_{55}$ & 
$E_1^t= t^3  e_1   $& $E_2^t= t^2 e_2 + t^5 e_4  $\\ 
$E_3^t= t^6 e_3 - t^6 e_4 - t^{12} e_5 $ & 
 $E_4^t = t^4 e_4 + t^{10} e_5 $ & $E_5^t= t^8 e_5   $\\
  
      \hline
$\cc^5_{49}(\A) \to \cc^5_{56}(\A)$ & 
$E_1^t= t^{-1}  e_1   $& $E_2^t= t^{-1}e_2  $\\ 
$E_3^t= t^{-2} e_3$ & 
 $E_4^t = t^{-2} e_4 $ & $E_5^t= t^{-4} e_5   $\\
  
    \hline
$\cc^5_{56}(1+t^{-4}) \to \cc^5_{57}$ & 
$E_1^t=   t e_1 + i t e_2   $& $E_2^t= e_2 $\\ 
$E_3^t=   t^2 e_3 - t^2  e_4$ & 
 $E_4^t =e_4 $ & $E_5^t=  e_5   $\\

     \hline
$\cc^5_{56}(0) \to \cc^5_{58}$ & 
$E_1^t=   t e_1   $& $E_2^t= t^2 e_2 $\\ 
$E_3^t=   t^2 e_3 $ & 
 $E_4^t = t^4e_4 $ & $E_5^t=  t^6 e_5   $\\

         \hline
  $\cc^5_{52}(2i) \to \cc^5_{59}$ & 
$E_1^t= i t^{-2} e_1 + t^{-2} e_2 - t^{-6} e_3 + t^{-6} e_4   $ & $E_2^t= t^{-1} e_2 $\\ 
$E_3^t=-t^{-4} e_3 + t^{-4} e_4 + 2t^{-8} e_5 $ & 
 $E_4^t =t^{-3} e_4 - t^{-7}    e_5  $ & $E_5^t= t^{-6} e_5    $\\

       \hline
$\cc^5_{52} \left(\frac{t -1}{\sqrt{t^2-1}} \right) \to \cc^5_{60}$ & 
$E_1^t=  (t^2 - 1)^{\frac{3}{2}} t^{-1} e_1 + (t - t^{-1}) e_2  $
&
 $E_2^t=  ( t^2 - 1) e_2 $ \\
 $E_3^t= ( t^2 - 1)^3 t^{-2} e_3 + (t^{-1} - t)^2 e_4 $  
 &
 $E_4^t = (t^{-1} - 2 t + t^3) e_4      $ & 
 $E_5^t= (t^2 - 1)^4 t^{-2} e_5    $\\
  
       \hline
$\cc^5_{52}(0) \to \cc^5_{61 }$ & 
$E_1^t= -i t^{-2} e_1 - t^{-2}     e_2   $&
 $E_2^t=  -t^{-1} e_2$\\ 
 $E_3^t= -t^{-4} e_3 + t^{-4} e_4   $ & 
 $E_4^t = t^{-3} e_4  $ & 
 $E_5^t= t^{-6} e_5    $\\
  
    \hline
$\cc^5_{52}\left( -\frac{8+t}{t} \right) \to \cc^5_{62 }(\A)$  
&\multicolumn{2}{l}{     $E_1^t= 2 e_1 + 2 e_2 + 2 (4 - 4 \A + t) t^{-2} e_3 + ( 8 \A-8 + 6 t)  t^{-2} e_4  $}\\

&$E_2^t= 4 t  e_2 + 2 e_3 - 2  e_4 $& 
 $E_3^t=4 e_3 + 4 e_4 + 256 (\A-1)t^{-3} e_5   $ \\& 
 $E_4^t =8 t e_4 + (8 - 32 \A t^{-1})  e_5  $ & 
 $E_5^t= 64 e_5    $\\

   \hline
  $\cc^5_{49}(1-t^{-4}) \to \cc^5_{63 }$ & 
    $E_1^t= i t^2 e_1 + i t e_2 + t^6 e_3 - t^6 e_4$
&
 $E_2^t= i t^2 e_2 - t^6 e_3 + t^6 e_4$\\ 
&\multicolumn{2}{l}{ $E_3^t= -t^4 e_3 + t^2 ( t^2-1) e_4 + t^3 (t^2-2) e_5   $}\\  
& $E_4^t=-t^3 e_4 - t^4 (1 + t^3) e_5  $ & 
 $E_5^t= t^5 e_5    $\\ 
  
     \hline
\multicolumn{2}{l}{ $\cc^5_{49}(1+t^2-5 t^4+10 t^6-10 t^8+5 t^{10}-t^{12}) \to \cc^5_{64 }$} & 
$E_1^t= -\frac{(-1)^{\frac{3}{4}}}{t^2-1} e_1 + (-1)^{\frac{1}{4}} e_2$\\
&
\multicolumn{2}{l}{ $E_2^t=(-1)^{\frac{1}{4}} t e_2 + \frac{i}{( t^2-1)^3} e_3 - \frac{i}{(t^2-1)^3} e_4  $}\\ 
&\multicolumn{2}{l}{ $E_3^t= -\frac{i}{(t^2-1)^2} e_3 + 
 i\left(t^{-2} + \frac{1}{(t^2-1)^{2}}\right) e_4 + \frac{2 - t^2}{3 t^2-1} e_5    $}\\  
& $E_4^t= it e_4 + \frac{t}{t^2-1} e_5 $ & 
 $E_5^t=-t^2 e_5    $\\ 
  
      \hline
  $\cc^5_{52}\left(\frac{1}{\sqrt{t^2-1}}\right) \to \cc^5_{65 }$ & 
    $E_1^t=-(t^2 - 1)^{\frac{3}{2}} t^{-1} e_1 + ( t-t{-1}) e_2 $\\
&
 $E_2^t= (   t^2-1) e_2$&
 $E_3^t= ( t^2-1)^3t^{-2}e_3 + ( t^{-2} -2 + t^2) e_4  $ \\ 
& $E_4^t= (t^{-1} - 2 t + t^3) e_4  $ & 
 $E_5^t= (  t^2-1)^4 t^{-2} e_5    $\\

        \hline
  $\cc^5_{52}( t^{-1} ) \to \cc^5_{66 }$ & 
 $E_1^t= t^5 e_1 - t^4 e_2$&
 $E_2^t=-t^5 e_2 + \frac{t^9}{2} e_3 + (t^7 - \frac{t^9}{2}) e_4 $\\ 
 $E_3^t= t^{10} e_3 + ( t^8 - t^{10}) e_4 $ &
 $E_4^t=t^9 e_4 $ & 
 $E_5^t=  t^{16} e_5    $\\
 
       \hline
  $\cc^5_{64}  \to \cc^5_{67}$ & 
 $E_1^t= t^{-1} e_1   $&
 $E_2^t= t^{-1} e_2  $\\
 $E_3^t= t^{-2} e_3  $ &
 $E_4^t=   t^{-2} e_4  $ & 
 $E_5^t=  t^{-4} e_5   $\\

          \hline
  $\cc^5_{56}( -t^{-4} ) \to \cc^5_{68 }$ & 
 $E_1^t= t e_1 + e_2  $&
 $E_2^t= t e_2  $\\
 $E_3^t= t^2 e_3 + (1 - t^2) e_4 $ &
 $E_4^t=   t e_4  $ & 
 $E_5^t=  t e_5   $\\

        \hline
$\cc^5_{69}\left(\frac{1+2 t^2+t^3}{t^4}\right) \to \cc^5_{70}$ & 
$E_1^t= t^{-1} e_1 + t^{-2} e_3 + t^{-1} e_4$&
$E_2^t=   t  e_2  $\\ 
$E_3^t=t^{-3} e_3 $ & 
$E_4^t =t^{-2} e_4 + (2 + t)t^{-3} e_5   $ & $E_5^t= t^{-4} e_5    $\\

        \hline
$\cc^5_{69}\left(\frac{1+t^3}{2 \sqrt[3]{2} t^3} \right) \to \cc^5_{71}$ & 
$E_1^t=\frac{1}{\sqrt[3]{2}t} e_1 + \frac{1}{2 t} e_3    $&
$E_2^t=  \sqrt[3]{2}   e_2 + \frac{1}{2} e_3  $\\ 
$E_3^t=\frac{1}{2 t^2} e_3 $ & 
 $E_4^t = \frac{1}{\sqrt[3]{2}t^2} e_4 + \frac{1}{\sqrt[3]{2} t^2} e_5   $ & $E_5^t= \frac{1}{2 \sqrt[3]{2}  t^4} e_5   $\\

        \hline
  $\cc^5_{69}(0) \to \cc^5_{73}$ & 
     $E_1^t= t^{-1} e_1 + t e_3   $&

 $E_2^t=  t^{-2} e_2 + t e_3  $\\ 
 $E_3^t=  e_3  $ & 
 $E_4^t =t^{-2}e_4     $ & $E_5^t= t^{-4} e_5   $\\
 
  \hline
  $\cc^5_{72}  \to \cc^5_{74}$ & 
     $E_1^t= e_1 + e_3   $&

 $E_2^t= t  e_2  $\\ 
 $E_3^t=  t^{-1} e_3  $ & 
 $E_4^t =   e_4    $ & $E_5^t=  e_5   $\\

        \hline
  $\cc^5_{74}  \to \cc^5_{75}$ & 
     $E_1^t=t^{-1} e_1 +  e_3   $&

 $E_2^t= e_2  $\\ 
 $E_3^t=  t^{-2} e_3  $ & 
 $E_4^t = t^{-2} e_4    $ & $E_5^t= t^{-4} e_5   $\\

   \hline
  $\cc^5_{80}(t^{-1})  \to \cc^5_{78}$ & 
     $E_1^t=2 (t-1) t e_1 + (1 - t) t e_2   $&\\

&\multicolumn{2}{l}{  $E_2^t=4 (t-1)^2 t^2 e_2 - 4 (t-1)^2 t^2 e_3 + (t-1)^2 t (1 + t) e_4 $}\\ 
&\multicolumn{2}{l}{ $E_3^t= 8 (t-1)^3 t^3 e_3 - 4 (t-1)^3 t^2 (3 + t) e_4 + 2 (t-1)^3 t^2 (3 + 7 t)  e_5$}\\ 
&  $E_4^t = 16 (t-1)^4 t^3 e_4 - 32 (t-1)^4 t^3 (1 + t)  e_5   $ & 
 $E_5^t= (t-1)^5 t^4 e_5  $\\

  \hline
    $\cc^5_{80}\left(\frac{3}{3+t}\right) \to \cc^5_{79}$ & 
\multicolumn{2}{l}{ $E_1^t= 4 t (3 + t) e_1 - 36 t (3 + t)  e_2 - 54 t (3 + t) (6 + t) e_3$}\\

& \multicolumn{2}{l}{$E_2^t=   16 t^2 (3 + t)^2 e_2 - 288 t^2 (3 + t)^2 e_3 + 
 11664 t^2 (3 + t) (24 + 10 t + t^2)  e_5 $}\\

& \multicolumn{2}{l}{$E_3^t= 64 t^3 (3 + t)^3 e_3 - 576 t^3 (3 + t)^2 (12 + t) e_4 - 
 2592 t^3 (3 + t)^2 (-24 - 2 t + t^2) e_5  $} \\
&  \multicolumn{2}{l}{$E_4^t= 768 t^4 (3 + t)^3 e_4 - 9216 t^4 (3 + t)^3 (6 + t) e_5  $} \\ 
&\multicolumn{2}{l}{$E_5^t=3072 t^5 (3 + t)^4 e_5$}\\
 
 \hline

\end{longtable}
}

By calculation of dimension of derivation algebra, we have dimensions ($\gdim$) of algebraic varieties  
defined by the following algebras:
\begin{longtable}{l}
$\gdim \ \cc_{49}^5(\alpha) = 24$\\
$\gdim \ \cc_{26}^5(\alpha,\beta) = 23$\\
$\gdim \ {\mathcal J}_{21}=\gdim \  \cc_{16}^5(\alpha)= \gdim \ \cc_{69}^5=\gdim \ \cc_{72}^5= \gdim \  \cc_{80}^5(\alpha)= \gdim \ \cc_{81}^5= 22$\\
$\gdim \ \cc_{76} = \gdim \ \cc_{77} =21$\\
\end{longtable}

Thanks to list of non-degeneration arguments presented below:
\begin{longtable}{|rcl|l|}
\hline 
\multicolumn{3}{|c|}{\textrm{Non-degeneration}} & \multicolumn{1}{|c|}{\textrm{Arguments}}\\
\hline
\hline

 $\cc_{26}^5(\alpha,\beta)$  & $\not \to$&   
 $\begin{array}{l} 
  \cc_{16}^5(\A),   \cc_{69}^5, \cc_{72}^5,  \cc_{76}^5, \\ 
  \cc_{77}^5,     \cc_{80}^5(\A) ,\cc_{81}^5, {\mathcal J}_{21}\end{array} $

 &  
${\mathcal R}=
\left\{ 
\begin{array}{l} A_1A_4=0 \end{array}  \right\} $
 \\
\hline
 
$\cc_{49}^5(\alpha)
 $  & $\not \to$&  
 
  $\begin{array}{l} 
  \cc_{16}^5(\A),   \cc_{26}^5(\A,\B), \cc_{69}^5, \cc_{72}^5,  \cc_{76}^5, \\ 
  \cc_{77}^5,     \cc_{80}^5(\A) ,\cc_{81}^5, {\mathcal J}_{21}\end{array} $

 &  
${\mathcal R}=
\left\{ 
\begin{array}{l} 
A_1^2 \subseteq A_3, A_1A_2 \subseteq A_4, \\
A_1A_3 \subseteq A_5,  
A_1A_5 \subseteq 0 
\end{array}  \right\} $

 \\
\hline

$\cc_{16}^5(\alpha)
 $  & $\not \to$&  
 
  $\begin{array}{l} 
   \cc_{76}^5,  
  \cc_{77}^5     \end{array} $ &  
${\mathcal R}=
\left\{ 
\begin{array}{l} 
A_1^2 \subseteq A_3, \ 
A_1A_2 \subseteq A_4, \
A_1A_3 \subseteq A_5  \end{array}  \right\} $ \\
\hline

$\cc_{80}^5(\alpha), \cc_{81}^5
 $  & $\not \to$&  
 
  $\begin{array}{l} 
   \cc_{76}^5,  
  \cc_{77}^5     \end{array} $ &  
${\mathcal R}=
\left\{ 
\begin{array}{l} 
A_3^2=0  \end{array}  \right\} $ \\
\hline

$\cc_{72}^5, \cc_{69}^5
 $  & $\not \to$&  
 
  $\begin{array}{l} 
   \cc_{76}^5,  
  \cc_{77}^5     \end{array} $ &  
${\mathcal R}=
\left\{ 
\begin{array}{l} 
A_1^2 \subseteq A_4  \end{array}  \right\} $ \\
\hline

${\mathcal J}_{21} 
 $  & $\not \to$&  
 
  $\begin{array}{l} 
   \cc_{76}^5,  
  \cc_{77}^5     \end{array} $ &  
${\mathcal R}=
\left\{ 
\begin{array}{l} 
{\mathcal J}_{21} \mbox{ is Jordan}  \end{array}  \right\} $ \\
\hline

\end{longtable}

we have that algebras 
\[ \Omega=\{  {\mathcal J}_{21}, \cc_{16}^5(\alpha),  \cc_{26}^5(\alpha,\beta),  \cc_{49}^5, \cc_{69}^5, \cc_{72}^5,  \cc_{76}, \cc_{77},  \cc_{80}^5(\alpha), \cc_{81}^5\}\]
give irreducible components.
\end{proof}

\end{document}